\documentclass[12pt, reqno]{amsart}

\usepackage{amsmath,amsfonts,amsthm,amssymb,amsxtra}
\usepackage{bbm} % to get \1
\usepackage{hyperref}	% to get hyperlinks in equations and references

\usepackage{tikz, caption}
\usetikzlibrary{shapes}

\newtheorem{theorem}{Theorem}%[section]

\newtheorem{lemma}[theorem]{Lemma}
\newtheorem{corollary}[theorem]{Corollary}

\theoremstyle{definition}

\newtheorem{definition}[theorem]{Definition}

\theoremstyle{remark}

\numberwithin{equation}{section}
\numberwithin{theorem}{section}

\renewcommand{\epsilon}{\varepsilon}

\newcommand{\N}{\mathbb{N}}

\renewcommand{\phi}{\varphi}
\newcommand{\R}{\mathbb{R}}

\newcommand{\Z}{\mathbb{Z}}

\begin{document}

\title[]{Rearrangement Inequalities on the lattice graph }

\author[]{Shubham Gupta}
\address{Department of Mathematics, Imperial College London, London, UK.} \email{s.gupta19@imperial.ac.uk}

\author[]{Stefan Steinerberger}
\address{Department of Mathematics, University of Washington, Seattle, WA 98195, USA.} \email{steinerb@uw.edu}

\thanks{S.G. is supported by President's Ph.D. scholarship, Imperial College London.}
\thanks{S.S. is supported by the NSF (DMS-2123224) and the Alfred P. Sloan Foundation.}

\begin{abstract}  The Polya-Szeg\H{o} inequality in $\mathbb{R}^n$ states that, given a non-negative function $f:\mathbb{R}^{n} \rightarrow \mathbb{R}_{}$, its spherically symmetric decreasing rearrangement $f^*:\mathbb{R}^{n} \rightarrow \mathbb{R}_{}$   is `smoother' in the sense of $\| \nabla f^*\|_{L^p} \leq \| \nabla f\|_{L^p}$ for all $1 \leq p \leq \infty$. We study analogues on the lattice grid graph $\mathbb{Z}^2$. The spiral rearrangement is known to satisfy the Polya-Szeg\H{o} inequality for $p=1$, the Wang-Wang rearrangement satisfies it for $p=\infty$ and no rearrangement can satisfy it for $p=2$. We develop a robust approach to show that both
these rearrangements satisfy the Polya-Szeg\H{o} inequality up to a constant for all $1 \leq p \leq \infty$. In particular, the Wang-Wang rearrangement satisfies  $\| \nabla f^*\|_{L^p} \leq 2^{1/p}   \| \nabla f\|_{L^p}$ for all $1 \leq p \leq \infty$. We also show the existence of (many) rearrangements on $\mathbb{Z}^d$ such that $\| \nabla f^*\|_{L^p} \leq c_d  \cdot \| \nabla f\|_{L^p}$ for all $1 \leq p \leq \infty$.
\end{abstract}

\maketitle

\section{Introduction} 
\subsection{Rearrangements.}
For continuous functions $f: \mathbb{R}^n \rightarrow \mathbb{R}_{\geq 0}$ the symmetric decreasing rearrangement refers to the process of  rearranging its level sets in such a way that they are preserved in volume while being radially centered around the origin. More formally, if $f: \mathbb{R}^n \rightarrow \mathbb{R}_{\geq 0}$
  is a nonnegative function vanishing at infinity, then its symmetric decreasing rearrangement $f^*:\mathbb{R}^n \rightarrow \mathbb{R}_{\geq 0}$ is the function for which 
    \begin{enumerate}
  \item the super-level sets $\left\{ x \in \mathbb{R}^n: f^*(x) \geq  s\right\}$ are balls centered at the origin 
  \item which have the same measure as the original super-level set
$$
 \left|\left\{ x \in \mathbb{R}^n: f(x) \geq  s\right\}\right| = \left|\left\{ x \in \mathbb{R}^n: f^*(x) \geq  s\right\}\right|.
$$
  \end{enumerate}
 This implies that $ \| f\|_{L^p} = \|f^*\|_{L^p}$ for all $p>0$.
The celebrated Polya-Szeg\H{o} inequality states that for all $1 \leq p \leq \infty$,
\begin{equation} \label{1.1}
 \| \nabla f^* \|_{L^p}  \leq  \| \nabla f \|_{L^p}.
  \end{equation}
  This inequality can be regarded as functional version of the isoperimetric inequality and, as such, has many applications in analysis, partial differential equations and mathematical physics, see Baernstein \cite{baernstein}, Lieb-Loss \cite{lieb2001analysis} and Polya-Szeg\H{o} \cite{faber_szego}. 
Inequalities of this type also appear naturally in variational problems and spectral geometry \cite{lieb1977existence, lieb_hardy, Luttinger_schrodinger, moser1971sharp,faber_szego, talenti1976best}. 
It is an interesting problem whether and to what extent inequalities like \eqref{1.1} can hold when the underlying space is a graph. Throughout this paper $G=(V,E)$ will represent a connected graph with countably infinite vertex set $V$ (which is usually indexed by $\mathbb{N}$). We further assume that every vertex of $G$ has finite degree. For $x, y \in V$, $x \sim y$ means that $(x, y) \in E$. There are natural analogues of $L^p-$spaces of functions and gradients on a graph and we define, for $1 \leq p < \infty$ and $f: V \rightarrow \mathbb{R}$, the norms
$$ \|f\|^p_{L^p} = \sum_{v \in V} |f(v)|^p \qquad \mbox{and} \qquad \| \nabla f\|_{L^p}^p = \sum_{x \sim y}|f(x) - f(y)|^p$$
with the usual modification at $p = \infty$.
\begin{definition}
Let $G=(V,E)$ be a graph. Let $f: V \rightarrow \R$ be a function vanishing at infinity, that is, $\{x \in V: |f(x)|>t\}$ is finite for all $t >0$. A \emph{rearrangement} $f^*$  is described by a bijective mapping $\eta:\mathbb{N} \rightarrow V$ (called \emph{labelling} of the vertex set $V$): for any given function $f$ one then defines
\begin{equation}
    f^*(\eta(k)) := k^{th} \hspace{5pt} \text{largest value assumed by} \hspace{5pt} |f|.
\end{equation}
\end{definition}
This definition of rearrangement follows the approach of Pruss \cite{pruss1998discrete}. Pruss also defined a particular labelling on regular trees called \emph{spiral-like labelling} (\cite[Definition 6.1]{pruss1998discrete}) and proved a discrete analogue of \eqref{1.1} for the infinite regular tree when $p=2$.
\begin{theorem}[Pruss \cite{pruss1998discrete}]\label{thm2.3}
Let $T_q $ be the infinite $q-$regular tree and let $f : V\rightarrow \R_{\geq 0}$ be a function vanishing at infinity and $f^*$ be the rearrangement of $f$ with respect to spiral-like labelling. Then
\begin{equation}\label{1.3}
    \sum_{x \sim y}|f^*(x)-f^*(y)|^2 \leq \sum_{x \sim y}|f(x)-f(y)|^2. 
\end{equation}
\end{theorem}

\begin{center}
  \begin{figure}[h!]
  \begin{tikzpicture}[scale=0.7]
\node at (0,0) {\includegraphics[width=0.4\textwidth]{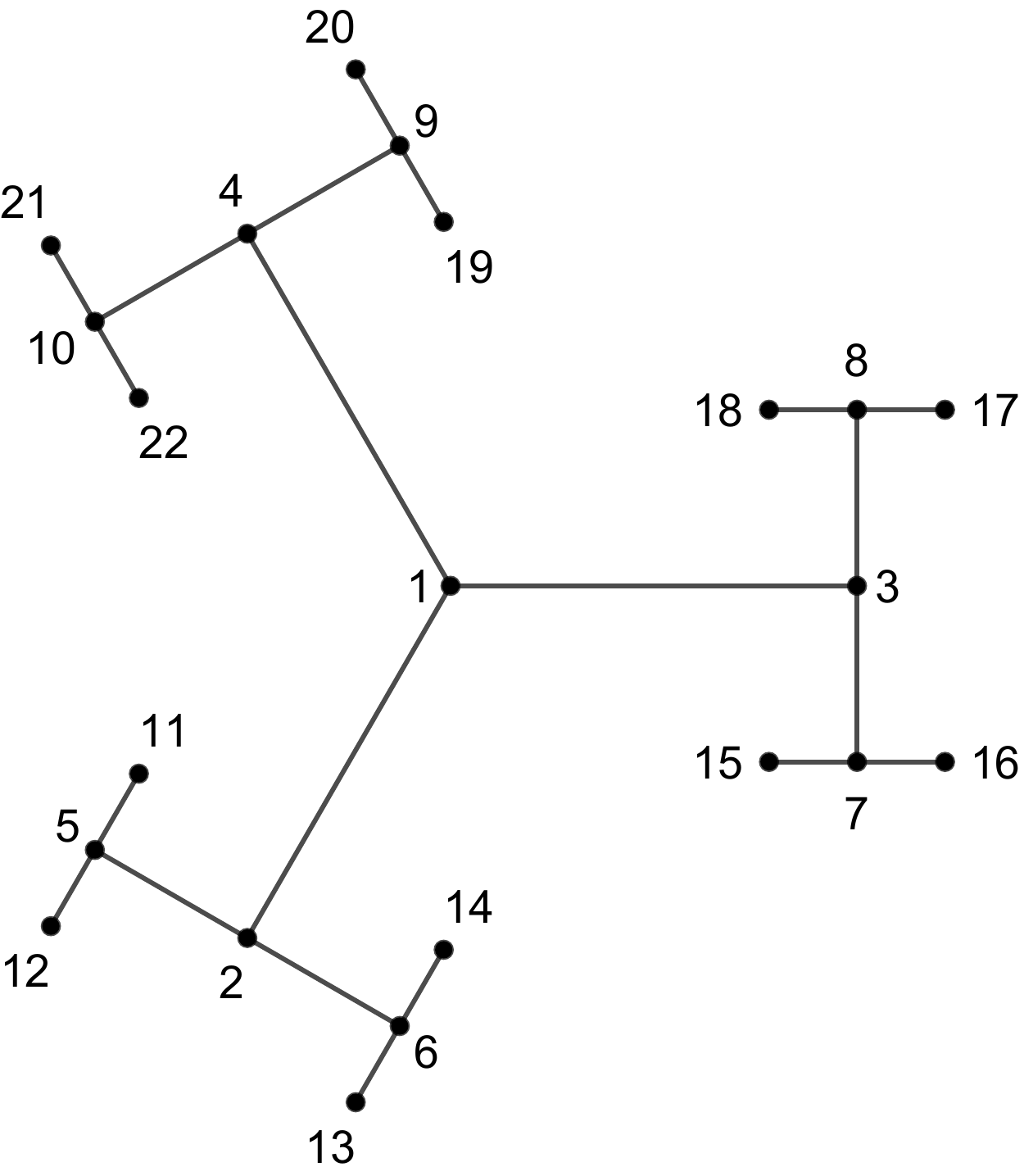}};
  \end{tikzpicture}
    \caption{The spiral-like labelling on the $3$-regular infinite tree (only the first three layers of the infinite tree are shown).}
    \label{fig:tree}
  \end{figure}
\end{center}
When $q=2$, the $2-$regular tree is simply the integer lattice graph on $\mathbb{Z}$, also known as the doubly-infinite path. In that case, the inequality has been extended from $L^2$ to $L^p$ for $p \geq 1$ by Hajaiej \cite{hajaiej2010rearrangement} and then extended by the first author to the weighted setting \cite{gupta2022symmetrization}. Recently, the second author \cite{steinerberger2022discrete} proved (\ref{1.3}) for all $p \geq 1$ and all $q-$regular trees.
The Polya-Szeg\H{o} inequality in $\mathbb{R}^n$ is a generalization of the isoperimetric inequality and relies on all the symmetries of the Euclidean space $\mathbb{R}^n$ and the special role played by the sphere. As such, it is well understood to be a rather special object and deeply tied to the underlying geometry. It is, for example, not at all clear how one would define rearrangement on a generic compact manifold. Likewise, the infinite regular tree is a highly symmetric object. A relevant question is whether something interesting can be said about rearrangements on more `generic' graphs like the lattice graph $\mathbb{Z}^d$.

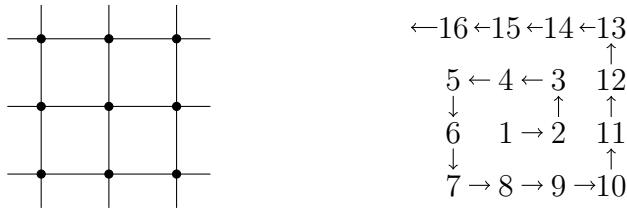
\begin{figure}[h!]
\begin{minipage}[left]{0.35\textwidth}
\begin{center}
\begin{tikzpicture}[scale=0.9]
    \filldraw (0.5, 0) circle (0.06cm);
    \filldraw (1.5, 0) circle (0.06cm);
    \filldraw (2.5, 0) circle (0.06cm);
    \filldraw (0.5, 1) circle (0.06cm);
    \filldraw (1.5, 1) circle (0.06cm);
    \filldraw (2.5, 1) circle (0.06cm);
    \filldraw (0.5, 2) circle (0.06cm);
    \filldraw (1.5, 2) circle (0.06cm);
    \filldraw (2.5, 2) circle (0.06cm);
    \draw  (0,0) -- (3,0);
    \draw  (0,1) -- (3,1);
    \draw  (0,2) -- (3,2);
    \draw  (0.5,-0.5) -- (0.5,2.5);
    \draw  (1.5,-0.5) -- (1.5,2.5);
    \draw  (2.5,-0.5) -- (2.5,2.5);
\end{tikzpicture}
\end{center}
\end{minipage}
\begin{minipage}[right]{0.35\textwidth}
\begin{center}
 \begin{tikzpicture}[scale=1.4]
    \node at (0,0) {1};
    \node at (0.5,0) {2};
    \node at (0.5,0.5) {3};
    \node at (0,0.5) {4};
    \node at (-0.5,0.5) {5};
    \node at (-0.5,0) {6};
    \node at (-0.5,-0.5) {7};
    \node at (0,-0.5) {8};
    \node at (0.5,-0.5) {9};
    \node at (1,-0.5) {10};
    \node at (1,0) {11};
    \node at (1,0.5) {12};
    \node at (1,1) {13};
    \node at (0.5,1) {14};
    \node at (0,1) {15};
    \node at (-0.5,1) {16};
    \draw [->] (1 ,-0.35) -- (1,-0.15);
    \draw [->] (1 ,0.15) -- (1,0.35);
    \draw [->] (1 ,1.3/2) -- (1,1.7/2);
    \draw [->] (1.7/2,2/2) -- (1.4/2,2/2);
    \draw [->] (0.7/2,2/2) -- (0.4/2,2/2);
    \draw [->] (-0.3/2,2/2) -- (-0.6/2,2/2);
    \draw [->] (-1.3/2,2/2) -- (-1.8/2,2/2);
    \draw [->] (0.15, 0) -- (0.35, 0);
    \draw [->] (1/2, 0.3/2) -- (1/2, 0.7/2);
    \draw [->] ( 0.7/2,1/2) -- (0.3/2,1/2);
    \draw [->] ( -0.3/2,1/2) -- (-0.7/2,1/2);
    \draw [->] ( -1/2,0.7/2) -- (-1/2,0.3/2);
    \draw [->] ( -1/2,-0.3/2) -- (-1/2,-0.7/2);
    \draw [->] ( -0.7/2,-1/2) -- (-0.3/2,-1/2);
    \draw [->] ( 0.3/2,-1/2) -- (0.7/2,-1/2);
    \draw [->] ( 1.3/2 ,-1/2) -- (1.7/2,-1/2);
\end{tikzpicture}
\end{center}
\end{minipage}
\caption{Left:  part of the standard lattice graph $(\mathbb{Z}^2, \ell^1)$. Right: spiral labelling on $\Z^2$.}
\label{fig:spiral}
\end{figure}

We will initially restrict ourselves to the case when graph $G$ is the standard lattice graph on $\Z^2$, that is,  two vertices $x, y \in \Z^2$ are connected by an edge if and only if $||x-y||_{\ell^1} =1$ (see Fig. \ref{fig:spiral}). 
Recently, it was shown that a Polya-Szeg\"o inequality of the type \eqref{1.3} cannot hold for the two dimensional lattice graph when $p=2$.
\begin{theorem}[Hajaiej, Han, Hua \cite{hajaiej2022discrete}]\label{thm1.4}
For every rearrangement on the lattice graph $(\mathbb{Z}^2, \ell^1)$ there exists a compactly supported function $f: \Z^2 \rightarrow \mathbb{R}_{\geq 0}$ such that
\begin{equation}\label{2.4}
    \sum_{x \sim y}|f^*(x)-f^*(y)|^2 > \sum_{x \sim y}|f(x)-f(y)|^2.
\end{equation}
\end{theorem}
This was made quantitative by the second author \cite{steinerberger2022discrete}: one can find a function $f$ supported on five vertices such that $\sum_{x \sim y}|f^*(x)-f^*(y)|^2 > 1.01\sum_{x \sim y}|f(x)-f(y)|^2$.
A natural and interesting question arises from this impossibility result.
\begin{quote}
\textbf{Problem.}
Let $1 \leq p \leq \infty$. Among all possible rearrangements on $\mathbb{Z}^2$, which one minimizes the constant $c_p \geq 1$ in the inequality
$$  \forall~f \in L^p(\mathbb{Z}^2) \qquad   \sum_{x \sim y}|f^*(x)-f^*(y)|^p \leq c_p \sum_{x \sim y}|f(x)-f(y)|^p~?$$
\end{quote}

This is motivated by trying to understand how well one can symmetrize functions on $\Z^2$. 
We believe that on the lattice graph there are two natural candidates for an extremal rearrangements: the spiral-like rearrangement and the Wang-Wang rearrangement. 
It was proven by the second author \cite{steinerberger2022discrete} that the spiral-like rearrangement is optimal for $p=1$. It is not difficult to see that the Wang-Wang rearrangement is optimal for $p=\infty$ (essentially  \cite[Theorem 3]{steinerberger2022discrete}). One of our main contributions is a robust framework to prove uniform boundedness for both rearrangements.

\subsection{Spiral Rearrangement.}
This rearrangement consists of writing the values attained by the function in a spiral-like decreasing fashion around a central vertex (see Figs. \ref{fig:spiral} and \ref{fig:example}). We show that the spiral rearrangement is uniformly bounded for $p>1$.

\begin{theorem}\label{thm1.5}
Let $f: \Z^2 \rightarrow \R_{}$ be a function vanishing at infinity and let $f^*$ denote the spiral rearrangement of $f$. Then, for  all $p \geq 1$,
\begin{equation}
    \left\|\nabla f^*\right\|_p \leq 4^{1+1/p} \left\|\nabla f\right\|_p.
\end{equation}
\end{theorem}
The constant $4^{1+1/p}$ is not sharp: it is known to be $1$ when $p=1$. It would be of interest to obtain improved bounds. A natural question is whether the optimality of the spiral rearrangement for $p=1$ extends beyond this endpoint: is it optimal in some range $p \in (1,p_0)$? The spiral rearrangement is natural because it is induced by a nested sequence of minimizers of the edge-isoperimetric problem (see  \cite{steinerberger2022discrete}).\\
  \vspace{-10pt}

    \begin{center}
  \begin{figure}[h!]
  \begin{tikzpicture}[scale=1]
  \node at (2,0.5) {\includegraphics[width=0.25\textwidth]{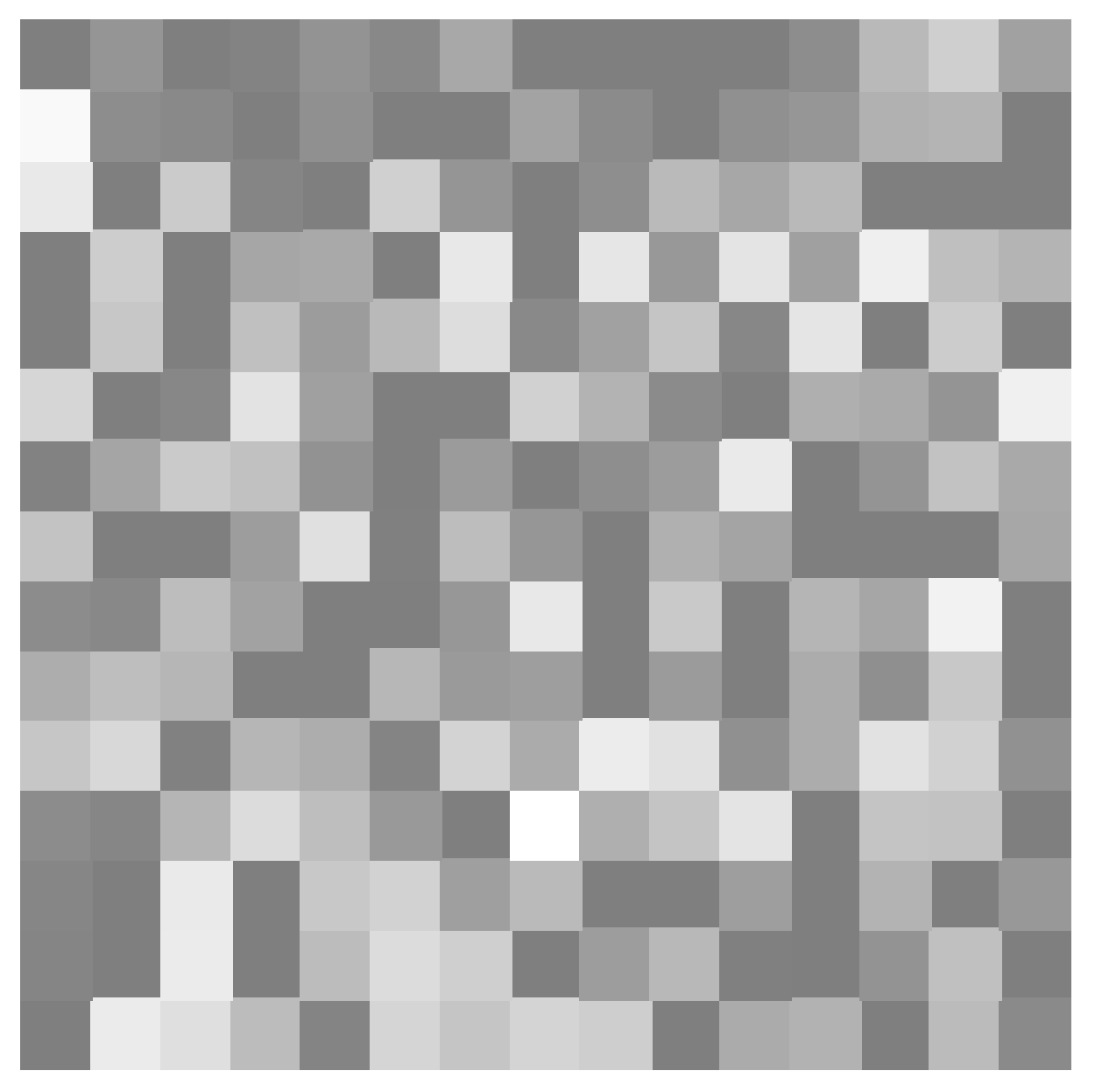}};
    \node at (7,0.5) {\includegraphics[width=0.25\textwidth]{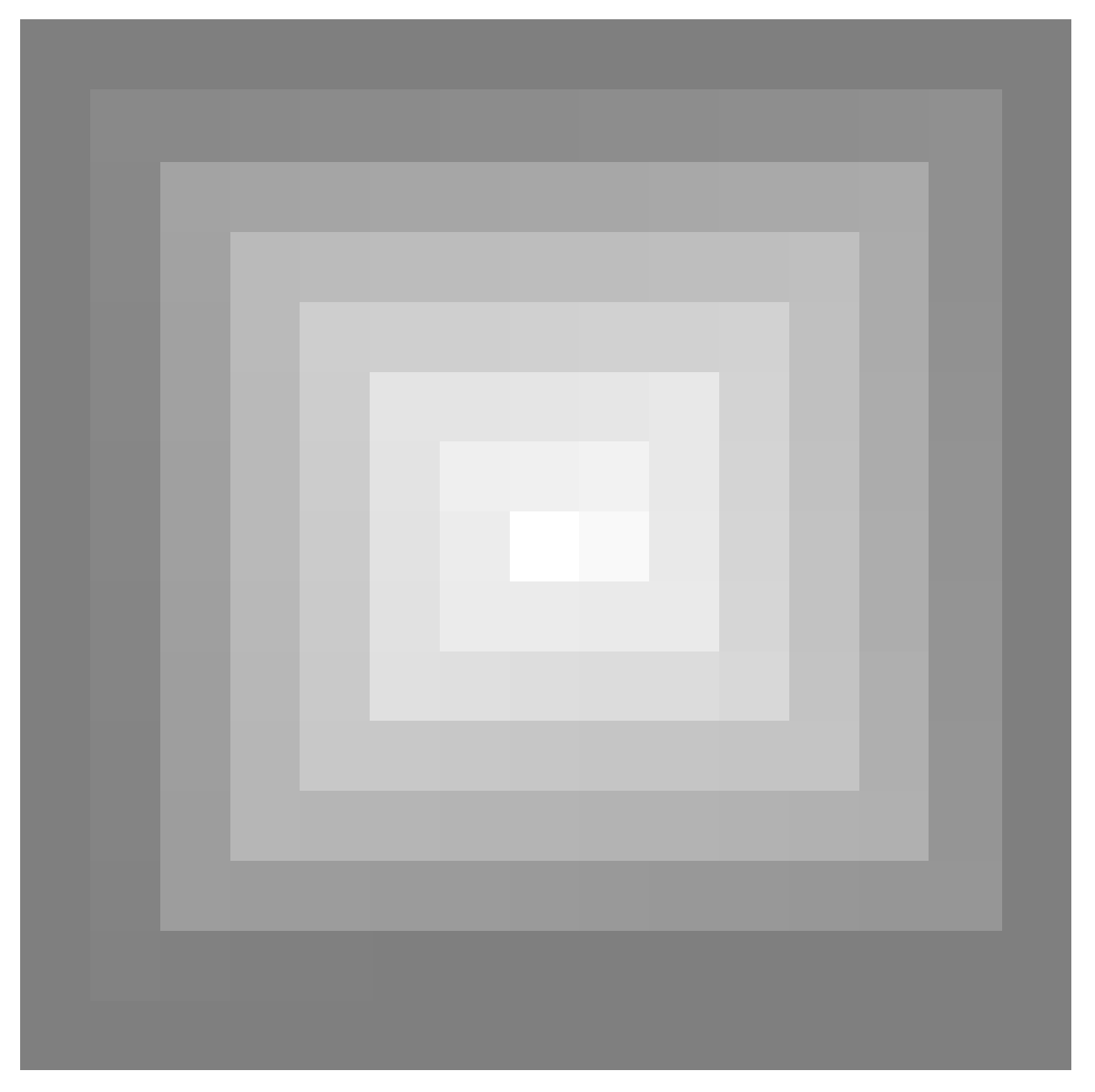}};
    \end{tikzpicture}
    \caption{Left: an example of a function $f:\mathbb{Z}^2 \rightarrow \mathbb{R}_{\geq 0}$ compactly supported around the origin (larger values are brighter). Right: the same function rearranged using the spiral rearrangement.}
      \label{fig:example}
  \end{figure}
  \end{center}
  \vspace{-20pt}

  \subsection{Wang-Wang Rearrangement.}
   The Wang-Wang rearrangement is naturally induced by a construction of nested solutions of the vertex-isoperimetric problem due to Wang \& Wang, see Fig. \ref{fig:smaln}. Taking the first $n$ vertices (with respect to the Wang-Wang label) leads to a set of $n$ vertices in $\mathbb{Z}^2$ that is adjacent to as few vertices as possible among all sets of $n$ vertices. It is not too difficult to see that this rearrangement satisfies $\| \nabla f^*\|_{L^{\infty}} \leq \| \nabla f\|_{L^{\infty}}$; this also follows from \cite[Theorem 3]{steinerberger2022discrete}. We obtain good uniform quantitative bounds for all $1 \leq p \leq \infty$.

\begin{theorem} \label{thm:wang} 
Let $f: \Z^2 \rightarrow \R_{}$ be a function vanishing at infinity and let $f^*$ denote the Wang-Wang rearrangement of $f$. Then, for  all $p \geq 1$,
\begin{equation}
    \left\|\nabla f^*\right\|_p \leq 2^{1/p} \left\|\nabla f\right\|_p.
\end{equation}
\end{theorem}

There is no reason to believe that the constant $2^{1/p}$ is sharp if $p<\infty$. Just as it is interesting whether the spiral rearrangement is optimal for $p \in (1, p_0)$ one could wonder whether the Wang-Wang rearrangement is optimal for $p \in (p_1, \infty)$. A particularly interesting question is whether $p_0 = p_1$: this would correspond to the theory of rearrangements on
$\mathbb{Z}^2$ having a particularly simple solution that interpolates between these two rearrangements. Or
are there other, yet unknown, rearrangements that are optimal in the intermediate range?
   
  \begin{figure}[h!]
\begin{minipage}[l]{.35\textwidth}
\begin{tikzpicture}
\node at (-3,0) {};
\node at (0,0) {1};
\node at (0,0.5) {2};
\node at (0.5,0) {3};
\node at (-0.5,0) {4};
\node at (0,-0.5) {5};
\node at (0.5,0.5) {6};
\node at (-0.5,0.5) {7};
\node at (0,1) {8};
\node at (1,0) {9};
\node at (0.5,-0.5) {10};
\node at (-1,0) {11};
\node at (-0.5,-0.5) {12};
\node at (0,-1) {13};
\end{tikzpicture}
\end{minipage}
\begin{minipage}[r]{.35\textwidth}
\begin{center}
\begin{tabular}{c  c  c  c c  c c}
$n$ & 1 & 2 & 3 & 4 & 5 & 6\\
$\partial_V^n$ &  4  &  6 &  7 & 8 & 8 & 9\\
\end{tabular}
\end{center}
\end{minipage} 
\caption{Left: start of the Wang-Wang enumeration (Wang \& Wang \cite{wang}). Right: $\partial_V^n$ , the size minimal vertex-boundary, for small $n$.} 
\label{fig:smaln}
\end{figure}

We note that the proof of Theorem \ref{thm:wang} is relatively concrete and there is some hope of generalizing it to higher dimensions (where a napkin computation would suggest the bound $d^{1/p}$ for $\mathbb{Z}^d$). Making this precise would require some additional combinatorial insights into the Wang-Wang enumeration of $(\mathbb{Z}^d, \ell^1)$.

\subsection{An abstract theorem.}
Our proof of Theorem \ref{thm1.5} is very concrete in each step and illustrates the core of the main argument. Abstracting the proof of Theorem \ref{thm1.5}, we arrive at the following general result.

\begin{theorem}\label{thm:abstract} Let $G=(V,E)$ be an infinite, connected graph with countable vertex set. Assume moreover that $\partial_V^{n+1} \geq \partial_V^{n} \geq \partial_V^1 \geq 2$ and that all vertices have uniformly bounded degree $\deg(v) \leq D$. Suppose now that $v_1, v_2, \dots$ is an enumeration of vertices with the property that, for some universal constant $c > 0$, we have
$$ \partial_V\left( \left\{v_1, \dots, v_n\right\} \right) \subseteq \left\{ v_{n+1}, \dots, v_{n + c \cdot \partial_V^n} \right\}.$$
Then for the rearrangement defined by this enumeration and all $1 \leq p \leq \infty$, we have
$$ \| \nabla f^*\|_{L^p} \leq  (c+1) \cdot D^{1/p} \cdot \| \nabla f\|_{L^p}.$$
\end{theorem}
This can be interpreted, at least philosophically, as an extension of  \cite[Theorem 3]{steinerberger2022discrete} from $L^{\infty}$ to $L^p$. One nice byproduct is that it allows us to show boundedness of a suitable (and very large class) of rearrangements in all lattice graphs $(\mathbb{Z}^d, \ell^1)$: 
we quickly define a notion of rearrangement on $\mathbb{Z}^d$ for which Theorem \ref{thm:abstract} implies uniform boundedness of the rearrangement. We say that an enumeration of the lattice points $\mathbb{Z}^d$
respects the $\ell^1-$norm if $\|v_i\|_{\ell^1} < \|v_j\|_{\ell^1}$ implies that $i < j$. This does not specify how to order lattice points with the same $\ell^1$ norm, for these any ordering is admissible. This defines a large number of possible enumerations.
\begin{corollary} \label{thm:cor} There exists a constant $c_d$ such that for every rearrangement respecting the $\ell^1-$norm and all $1 \leq p \leq \infty$
$$  \| \nabla f^*\|_{L^p(\mathbb{Z}^d)} \leq  c_d \cdot \| \nabla f\|_{L^p(\mathbb{Z}^d)}$$
\end{corollary}

This may look a priori like a strong result in so far as it applies uniformly to all $\ell^p-$spaces of functions as well as a very large number of possible enumerations of the vertex set. The price we pay is a lack of control on the constant $c_d$. It could also be interpreted as saying that our present approach is not fine enough to distinguish optimal rearrangements from merely very good rearrangements. As mentioned above, one could hope for a bound along the lines of $c_d \leq d^{1/p}$ from the Wang-Wang enumeration in higher dimensions and this is an interesting problem. We also mention that one could conceivably extend the definition of enumerations respecting the $\ell^1-$norm to enumerations respecting more general norms $\| \cdot \|_X$ on $\mathbb{R}^d$ without changing too much in the proof of Corollary \label{thm:cor} but we will not pursue this here.

\subsection{Structure of the paper.} 
The rest of the paper is structured as follows: Theorem \ref{thm1.5} is, in a suitable sense, the most `generic' application of our framework (explicit constants, no magic simplifications): we will spend most of the paper building a framework that allows us to prove it. The other results tend to follow from the same framework and, using the framework, have shorter proofs. More concretely, the remainder of the paper is structured as follows.

\begin{enumerate}
\item In \S \ref{sec:comparsiongraph}, we construct what we call the \emph{universal comparison graph} $G_c$. This graph is an infinite tree whose structure is built from the structure of solutions of the (vertex)-isoperimetric problem on $G$. There is way of mapping functions $f$ on $G$ to functions on $G_c$ in a way that decreases the $L^p-$norm of their gradient. 
\item \S \ref{sec:specific} discusses the universal comparison graph attached to $(\mathbb{Z}^2, \ell^1)$.
\item \S \ref{sec:mapping} is concerned with a mapping of edges in $(\mathbb{Z}^2, \ell^1)$ to short paths in its universal comparison graph $G_c$. We show that while this cannot be done bijectively, there is a mapping $\Psi$ of edges from the lattice graph to the comparison graph that has bounded multiplicity (in the sense of the cardinality of the pre-image of an edge in $G_c$ being uniformly bounded by a universal constant).
\item \S  \ref{sec:mainproof}  uses the results from the previous sections to prove Theorem \ref{thm1.5}.
\item \S \ref{sec:wang} is dedicated to proving Theorem \ref{thm:wang}. The proof can be seen as a particularly simple application of the framework
developed above (for example, $\Psi$ is mapping edges to edges instead of mapping edges to paths).
\item \S \ref{sec:last} gives a proof of Theorem \ref{thm:abstract} and Corollary \ref{thm:cor}. The proof of Theorem \ref{thm:abstract} is essentially identical to the proof of  Theorem \ref{thm1.5} and relies heavily on \S 2 while replacing \S \ref{sec:specific}  and \S \ref{sec:mapping} with more abstract conditions. Corollary \ref{thm:cor} follows quickly from Theorem \ref{thm:abstract} (at the cost of giving no control on the constant).
\end{enumerate}

\section{The Universal Comparison graph}\label{sec:comparsiongraph}
The purpose of this section is to introduce the universal comparison graph: it appears to be a useful concept when studying rearrangements on graphs ( `universal' refers it being independent of the function and the rearrangement, it only depends on the graph itself.). In this section, we
develop a bit of abstract theory for general graphs with Lemma \ref{lem2.5} being the main goal. In the next section we will specify the behavior of the universal comparison graph when $G=(\mathbb{Z}^2, \ell^1)$, in that case the universal comparison graph is a completely explicit infinite tree without leaves (see Fig. \ref{fig:ucg0}).

\subsection{Definition}  Before introducing the universal comparison graph, we quickly introduce some of the relevant concepts (none of which are new).
Let $G=(V, E)$ be a connected graph with countably infinite vertex set $V$. For any set of vertices $X \subseteq V$ the \emph{vertex boundary} of $X$ is defined as
\begin{align*}
    \partial_V(X) := \{z \in V\setminus X: x \sim z,   \text{for some}\hspace{3pt} x \in X \},
\end{align*}
and $|\partial_V(X)|$ is the \emph{vertex perimeter} of set $X$. We define \emph{isoperimetric number} $\partial_V^n$ as the solution of vertex isoperimetric problem on $G$ among all sets with $n$ vertices 
\begin{equation}
    \partial_V^n := \min_{X \subseteq V, \\ |X|=n}  |\partial_V(X)|,
\end{equation}
for $n \in \N$. We note that on $(\mathbb{Z}^2, \ell^1)$ the vertex-isoperimetric problem is completely solved, we refer to Wang \& Wang \cite{wang}. We also refer to 
\cite{boll, boll2, gar, harper, lindsey} and references therein for related results.
We are now ready to define the main object of this section. 

\begin{definition}\label{def2.1}
Let $G=(V, E)$ be a graph, let $\partial_V^n$ be its isoperimetric number and assume that $\partial_V^{n+1} \geq \partial_V^n$ for $n \geq 1$. Then the \emph{universal comparison graph} $G_c =(\N, E_c)$ is the unique graph on $\N$ satisfying
\begin{equation}\label{2.2}
    \{i \in \N: i > n, i \sim n\} = \{i \in \N: (n-1)+\partial_V^{n-1} < i \leq n+\partial_V^n\},
\end{equation}
for all $n \geq 1$ with the convention $\partial_V^0 = 1$. 
\end{definition}

We quickly explain the construction for the graph $(\mathbb{Z}^2, \ell^1)$ by appealing to results of Wang \& Wang \cite{wang}.
They construct a permutation of the vertices such that the first $n$ vertices corresponding to that permutation minimize the vertex perimeter (the number of adjacent vertices) among all subsets of size $n$ uniformly in $n$.  
Note that this is something special: for most graphs one cannot expect the solutions of the vertex-isoperimetric problem to be
nested, one would expect that they vary a great deal depending on the number of vertices under consideration. Appealing to the
definition above, we see that the universal comparison graph is going to be a graph on $\mathbb{N} = \left\{1,2,3,\dots,\right\}$. Plugging
in $n = 1$, we see that $1$ is adjacent to $\left\{2,3,4,5\right\}$. Plugging in $n=2$, we see that the vertex $2$ is adjacent to $\left\{5 < i \leq 8\right\}$.
The pattern continues, Figs. \ref{fig:ucg0} and \ref{fig:ucg} show the first few levels of the universal comparison tree.

\begin{lemma}\label{lem2.2}
The universal comparison graph $G_c$ is well defined, that is, there is exactly one graph on $\N$ satisfying \eqref{2.2}.
\end{lemma}

\begin{proof}
Let $n \in \N$, then condition \eqref{2.2} fixes all neighbours of $n$ which are greater than $n$. Next, we prove that \eqref{2.2} also fixes neighbours of $n$ which are smaller than $n$, thereby proving the uniqueness. In particular, we prove that for each vertex $n \geq 2$, there exists exactly one vertex $i <n$ such that $i \sim n$. Since $m+\partial_V^m$ is a strictly increasing sequence of integers, for each $n \geq 2$, there exists $1 \leq i<n$ such that $(i-1)+\partial_V^i < n \leq i+\partial_V^i$. Then condition \eqref{2.2} imply that $i \sim n$ and the same argument, monotonicity of $m+\partial_V^m$, establishes uniqueness. This proves that each vertex $n \geq 2$ has exactly one neighbour smaller than $n$.  
\end{proof}

\begin{center}
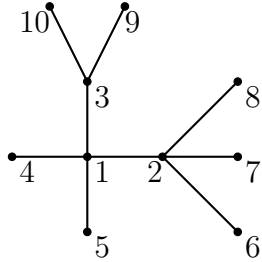
\begin{figure}[h!]
\begin{tikzpicture}
\filldraw (0,0) circle (0.05cm);
\node at (0.2, -0.2) {1};
\filldraw (1,0) circle (0.05cm);
\node at (0.9, -0.2) {2};
\filldraw (0,1) circle (0.05cm);
\node at (0.2, 1-0.2) {3};
\filldraw (-1,0) circle (0.05cm);
\node at (-1+0.2, -0.2) {4};
\filldraw (0,-1) circle (0.05cm);
\node at (0.2, -1-0.2) {5};
\draw [thick] (-1,0) -- (1,0);
\draw [thick] (0,-1) -- (0,1);
\filldraw (2,1) circle (0.05cm);
\filldraw (2,0) circle (0.05cm);
\filldraw (2,-1) circle (0.05cm);
\node at (2.2, 1-0.2) {8};
\node at (2.2, 0-0.2) {7};
\node at (2.2, -1-0.2) {6};
\draw [thick] (1,0) -- (2,1);
\draw [thick] (1,0) -- (2,0);
\draw [thick] (1,0) -- (2,-1);
\filldraw (0.5,2) circle (0.05cm);
\filldraw (-0.5,2) circle (0.05cm);
\draw [thick] (0,1) -- (0.5, 2);
\draw [thick] (0,1) -- (-0.5, 2);
\node at (0.6, 1.8) {9};
\node at (-0.7, 1.8) {10};
\end{tikzpicture}
\caption{Initial segment of the universal comparison graph for $(\mathbb{Z}^2, \ell^1)$.}
\label{fig:ucg0}
\end{figure}
\end{center}
\vspace{-0pt}

\begin{lemma}\label{lem2.3}
The universal comparison graph $G_c$ is a tree. Furthermore if $\partial_V^1 \geq 2$ then $G_c$ has no leaves, that is, each vertex has degree at least two.
\end{lemma}

\begin{proof}
First we prove that $G_c$ is connected. This follows from the argument above showing that each vertex $n$ is adjacent to exactly one vertex smaller than $n$. This induces a path to the vertex 1. Since each vertex is connected to the vertex 1, the graph is connected.
In Lemma \ref{lem2.2} we proved that each vertex $n \geq 2$ has exactly one neighbour smaller than $n$. This immediately implies that $G_c$ has no cycles: for any cycle, the largest vertex in the cycle say $l$ will have at least two neighbours which are smaller than $l$, which is not possible. Therefore $G_c$ is a tree.
Let $n \geq 2$, since $m+\partial_V^m$ is a strictly increasing sequence, condition \eqref{2.2} implies that $|\{i \in \N: i > n, i \sim n \}| \geq 1$. From Lemma \ref{lem2.2} we also know that each vertex $n \geq 2$ has exactly one neighbour smaller that $n$. This proves that degree of vertex $n$ is at least two, for $n \geq 2$. It is clear that  the degree of vertex $1$ is $\partial_V^1 \geq 2$. Therefore $G_c$ has no leaves.  
\end{proof}

In the next lemma, we compute the vertex boundary of first $n$ vertices of $G_c$.
\begin{lemma}\label{lem2.4}
Let $G$ be a graph and $G_c$ be its universal comparison graph. Then 
\begin{equation}\label{2.3}
    \partial_V(\{1,2,...,n\}) = \{n+1, n+2,..., n+\partial_V^n\}.
\end{equation}
\end{lemma}

\begin{proof}
We prove the result using induction on $n$. Let us assume that \eqref{2.3} holds true for $n \geq 1$. It is easy to see  that
$$\partial_V(\{1,2,...,n, n+1\}) = \partial_V(\{1,2,...,n\})\setminus\{n+1\} \cup \{i \in \N: i > (n+1), i \sim n+1\}.$$
Above identity along with \eqref{2.2} proves \eqref{2.3} for $n+1$. Identity \eqref{2.3} for $n=1$ follows from \eqref{2.2},
$$ \partial_V(\{1\}) = \{i \in \N: i >1, i \sim 1\} = \{2,..,1+\partial_V^1\}.$$
\end{proof}

\subsection{Main Lemma.} We are now ready to prove the main result of this section. Let $f: V(G) \rightarrow \R_{\geq 0}$ be a function vanishing at infinity. We define a function $f_c: V_c \rightarrow \R$ on the vertices of its universal comparison graph $G_c = (V_c, E_c)$ as 
$$ f_c(k) := k^{th} \hspace{3pt} \text{largest value attained by} \hspace{3pt} |f|.$$
We will refer to $f_c$ as the \emph{comparison function} of $f$. We will now show that the comparison function has a smaller gradient in the sense
of   
$$ \left\| \nabla f_c\right\|_{L^p(G_c)} \leq    \left\| \nabla f\right\|_{L^p(G)}.$$ 
\begin{lemma}[Comparison Lemma]\label{lem2.5}
Let $G= (V,E)$ be a graph and $f: V \rightarrow \R_{\geq 0}$ be a function vanishing at infinity. Then for $p \geq 1$,
\begin{equation}
    \sum_{(x, y) \in E_c} |f_c(x)-f_c(y)|^p \leq \sum_{(x,y) \in E} |f(x)-f(y)|^p.
\end{equation}
\end{lemma}

 The proof is based on a coarea formula already used in \cite{steinerberger2022discrete} which we quickly explain for the convenience of the reader.
 The coarea formula on graphs is
$$ \| \nabla f\|_{L^p}^p =  \int_{0}^{\infty} \int_{ \partial_{E} \left\{ f \geq s \right\} } \left| \nabla f\right|^{p-1} dx ds,$$
where $ \partial_{E} \left\{ f \geq s \right\}$ is the set of edges that connect the vertex sets $\left\{ v \in V: f(v) \geq s \right\}$ and
$\left\{ v \in V: f(v) < s \right\}$. It is easily derived: the idea being each edge contributes $|f(v) - f(w)|^p$ to the
left-hand side and $|f(v) - f(w)|^{p-1}$ over an interval of length $|f(v) - f(w)|$ to the right-hand side. We will now use a small modification
of the idea: the advantage of this new formulation is that the values in $\left\{ f \geq s \right\}$ no longer show
up in the inner integral which is solely determined by $s$ and the values outside.

\begin{lemma}[Modified Coarea Formula, see \cite{steinerberger2022discrete}] Let $1 \leq p < \infty$. Then
$$ \| \nabla f\|_{L^p}^p =  p\int_{0}^{\infty} \int_{ \partial_{E} \left\{ f \geq s \right\} } \left| \nabla \min(f, s) \right|^{p-1} dx ds.$$
\end{lemma}
There is a quick proof: consider again a single edge $(v,w) \in E$. The contribution to the left-hand side is
$|f(v) - f(w)|^p$. 
The contribution to the right-hand side is
$$ \int_{\min\left\{f(v), f(w) \right\}}^{\max\left\{f(v), f(w) \right\}} \left(s - \min\left\{f(v), f(w) \right\}\right)^{p-1} ds = \frac{|f(v) - f(w)|^p}{p}.$$

\begin{proof}[Proof of Lemma \ref{lem2.5}]
We can assume, without loss of generality, that $||f||_\infty =1$. The modified coarea formula allows us to rewrite 
$$X =   \sum_{x \sim  y \in V} |f(x)-f(y)|^p$$
as
\begin{equation}\label{2.5}
  X =  p \int_0^{1} \sum_{(x,y) \in E(\{f>s\}, \{f>s\}^c)}|\min(s, f(x))-\min(s,f(y))|^{p-1} ds,
\end{equation}
where $E(X,Y)$ denotes the set of edges between $X, Y \subseteq V(G)$. We argue that the desired integral is monotone for each fixed $0 < s< 1$. 
Let us thus fix a value of $0 < s < 1$ and consider the quantity
$$ Y =     \sum_{(x,y) \in E(\{f>s\}, \{f>s\}^c)}|\min(s, f(x))-\min(s,f(y))|^{p-1}.$$
There is a naturally associated integer $i$ defined via
$$ f_c(1) \geq f_c(2) \geq \dots \geq  f_c(i) \geq s > f_c(i+1) \geq f_c(i+2) \geq \dots$$
The sum $Y$ is then a sum running over all edges connecting $\left\{f \geq s \right\}$ and $\left\{f < s \right\}$. We note that $\left\{f \geq s \right\}$ is
finite and has a number of neighbors is at least as big as $\partial_V^i$ (the smallest number of vertices that \textit{any} set of $i$ vertices
is adjacent to). Therefore
$$ Y \geq  \sum_{y \in \partial_V(\{f>s\})}|s-f(y)|^{p-1}.$$
We do not have too much information about $ \partial_V(\{f>s\})$ but certainly the sum is smallest if the values are as close as possible to $s$. Thus,
$$ \sum_{y \in \partial_V(\{f>s\})}|s-f(y)|^{p-1} \geq  \sum_{j=1}^{\partial_V^i} |s-f_c(i+j)|^{p-1}.$$
Using \eqref{2.3} and the fact that $G_c$ is a tree we obtain 
$$\sum_{j=1}^{\partial_V^i} |s-f_c(i+j)|^{p-1} = \sum_{(x,y) \in E(\{f_c>s\}, \{f_c>s\}^c)}|\min(s, f_c(x)) -\min(s,f_c(y))|^{p-1}.
    $$
Integrating over all $s$ and applying the modified coarea formula once more
$$ p\int_0^1 \sum_{(x,y) \in E(\{f_c>s\}, \{f_c>s\}^c)}|\min(s, f_c(x)) -\min(s,f_c(y))|^{p-1} ds = \| \nabla f_c\|_{L^p(G_c)}^p.$$
\end{proof}

\section{The Universal Comparsion Graph of $(\mathbb{Z}^2, \ell^1)$} \label{sec:specific}
For the remainder of the paper, we will be interested in the universal comparison graph of the lattice graph on $(\Z^2, \ell^1)$. In the next lemma, we prove that universal comparison graph of the lattice graph is well defined by proving $ \partial_V^{n+1} \geq \partial_V^n$.

\begin{lemma}\label{lem2.6}
Let $G$ be the lattice graph and $\partial_V^n$ be its isoperimetric number. Then, for all $n \geq 1$, we have
$ \partial_V^{n+1} \geq \partial_V^n.$
Furthermore, if $n = 2k(k+1)+1$, then
\begin{equation}\label{2.6}
    \partial_V^n = 4k+4. 
\end{equation}
\end{lemma}

\begin{proof}
In \cite{wang} Wang and Wang studied the vertex isoperimetric problem on the lattice graph. They constructed an enumeration of the vertices such that set of vertices with label $\leq n$ has as few vertices as any set of $n$ elements: this correspond to the construction of a nested sequence of extremizers. We recall the labelling (which was already shown above) in Fig. \ref{fig:sain}. One notably feature is that the nested sequence fills up $\ell^1-$balls in a layer-by-layer fashion.
More precisely, if $x,y \in \mathbb{Z}^2$ and if $\|x\|_{\ell^1} < \|y\|_{\ell^1}$ then the label of $x$ is smaller than the label of $y$. This proves that if $n= 2k(k+1)+1$ for $k \in \N$ (this is the size of $\ell^1$ closed ball in $\Z^2$ of radius $k$) then
$$ \partial_V^n = |\partial_V(\{x \in \Z^2: \left\|x\right\|_{\ell^1} \leq k\})| = 4k+4.$$
It remains to prove that $\partial_V^{n+1} \geq \partial_V^n$. Let $2k(k+1)+1 \leq n < 2(k+1)(k+2)+1 $, for non-negative integer $k$. Assume that $\partial_V^{n+1} < \partial_V^n$. We take the set of all points with label $\leq n+1$. We start by noting that corners of $\ell^1$ ball of radius $k+1$ cannot lie in the set: if a corner lies in the set, then removing the corner does not increase the vertex boundary (see Fig. \ref{fig:sain}) which shows  $\partial_V^{n} \leq \partial_V^{n+1}$ contradicting $ \partial_V^{n+1} < \partial_V^n$. Let us now consider all those vertices whose $y-$coordinate is maximal among all the $n+1$ points in the set. Let $(x,y)$ denote one such vertex with minimal $x$-coordinate. Then removing $(x,y)$ from the set does not increase the vertex boundary since $(x,y)$ is adjacent to $(x,y+1)$ which is not adjacent to any other vertex in the set; removing $(x,y)$ removes this neighbor while only adding $(x,y)$ as a new neighbor. This contradicts $\partial_V^{n+1} < \partial_V^n$. 
\end{proof}
\vspace{-10pt}
 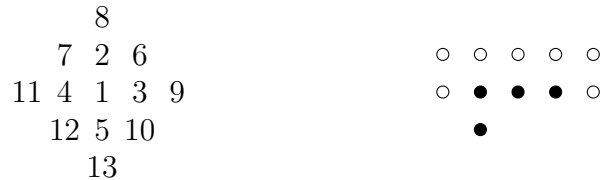
\begin{figure}[h!]
\begin{minipage}[l]{.43\textwidth}
\begin{tikzpicture}
\node at (-3,0) {};
\node at (0,0) {1};
\node at (0,0.5) {2};
\node at (0.5,0) {3};
\node at (-0.5,0) {4};
\node at (0,-0.5) {5};
\node at (0.5,0.5) {6};
\node at (-0.5,0.5) {7};
\node at (0,1) {8};
\node at (1,0) {9};
\node at (0.5,-0.5) {10};
\node at (-1,0) {11};
\node at (-0.5,-0.5) {12};
\node at (0,-1) {13};
\end{tikzpicture}
\end{minipage}
\begin{minipage}[r]{.25\textwidth}
\begin{center}
\begin{tikzpicture}
\filldraw (0,0) circle (0.08cm);
\draw (-0.5,0) circle (0.08cm);
\draw (0,0.5) circle (0.08cm);
\filldraw (0.5,0) circle (0.08cm);
\filldraw (1,0) circle (0.08cm);
\draw (1.5,0.5) circle (0.08cm);
\draw (1.5,0) circle (0.08cm);
\draw (1,0.5) circle (0.08cm);
\draw (0.5,0.5) circle (0.08cm);
\draw (-0.5,0.5) circle (0.08cm);
\filldraw (0,-0.5) circle (0.08cm);

\end{tikzpicture}
\end{center}
\end{minipage} 
\caption{Nested minimizers of the vertex-isoperimetry problem (Wang \& Wang \cite{wang}) (left) and a step in the proof.} 
\label{fig:sain}
\end{figure}

 Lemma \ref{lem2.6} implies useful geometric information about the comparison graph of $(\mathbb{Z}^2, \ell^1)$. We use 
 $ S(r;G_C) :=\{i \in V_c : d_{G_c}(i,1) = r\},$
and 
$ B(r;G_C) := \{i \in V_c : d_{G_c}(i,1) \leq r\}$
to denote the sphere and closed ball of radius $r$ in $G_c$. 
\begin{lemma} \label{lem:diam}
Let $G = (\mathbb{Z}^2, \ell^1)$ and $G_c$ be its universal comparison graph. Then
\begin{equation}
  \forall~r \in \mathbb{N}_{\geq 1} \qquad   |S(r; G_c)| = 4r \hspace{9pt} \text{and} \hspace{9pt} |B(r;G_c)| = 1+2r(r+1).
\end{equation}
\end{lemma}
 This follows immediately from the existence of nested minimizers (Wang \& Wang \cite{wang}), see also Fig. \ref{fig:ucg0}. It implies that 
$$\{i \in V_c : d_{G_c}(i,1) = r\} = \left\{v \in \mathbb{Z}^2: \|v\|_{\ell^1} = r \right\}$$
from which one deduces $|S(r; G_c)| = 4r$ and then, by summation,
$$  |B(r;G_c)| = 1 + \sum_{i=1}^{r} 4r = 1 + 2r(r+1).$$

\section{Embedding Edges into Universal Comparison graph}\label{sec:mapping}
For the purpose of this section we will assume again that $G= (\Z^2, \ell^1)$ is the lattice graph on $\Z^2$ and that $G_c$ is the associated universal comparison graph. In this section, we assume that vertices of $G$ are labelled using the spiral labelling on $\Z^2$, as defined in Figure \ref{fig:spiral}. The rest of this section is devoted to the study of the map
$$ \Psi : E_{(\mathbb{Z}^2, \ell^1)} \rightarrow \{\text{paths in}\hspace{3pt} G_c\},$$
defined as follows: Let $i<j$ and $(i,j) \in E_{(\mathbb{Z}^2, \ell^1)}$ be an edge in the lattice graph. Then $\Psi(i, j)$ is defined as the shortest path in the tree $G_c$ between the vertex $i$ and smallest vertex $k \in V_c$ such that
\begin{enumerate}
\item $k \geq j$ 
\item and $i$ lies in the path between $1$ and $k$ in $G_c$. 
\end{enumerate}
The second condition could also be phrased as follows: since $G_c$ is a tree with no leaves (Lemma \ref{lem2.2}), we may think of the vertex 1 as a root. Then the vertex $i$ has infinitely many descendants (among which we pick the smallest one, $k$, that is at least as big as $j$). We will show that $\Psi$ maps edges to paths with uniformly bounded length.
 \begin{figure}[h!]
\begin{minipage}[l]{.3\textwidth}
\begin{tikzpicture}
\node at (-3,0) {};
\node at (0,0) {1};
\node at (0,0.5) {2};
\node at (0.5,0) {3};
\node at (-0.5,0) {4};
\node at (0,-0.5) {5};
\node at (0.5,0.5) {6};
\node at (-0.5,0.5) {7};
\node at (0,1) {8};
\node at (1,0) {9};
\node at (0.5,-0.5) {10};
\node at (-1,0) {11};
\node at (-0.5,-0.5) {12};
\node at (0,-1) {13};
\end{tikzpicture}
\end{minipage}
\begin{minipage}[l]{.53\textwidth}
\begin{tikzpicture}
\node at (-3,0) {};
\filldraw (0,0) circle (0.06cm);
\filldraw (0,0.5) circle (0.06cm);
\filldraw (-0.5,0) circle (0.06cm);
\filldraw (0,-0.5) circle (0.06cm);
\filldraw (0.5,0.5) circle (0.06cm);
\filldraw (-0.5,0.5) circle (0.06cm);
\filldraw (0,1) circle (0.06cm);
\filldraw (1,0) circle (0.06cm);
\filldraw (0.5,-0.5) circle (0.06cm);
\filldraw (-0.5,-0.5) circle (0.06cm);
\filldraw (-1,0) circle (0.06cm);
\filldraw (0,-1) circle (0.06cm);
\filldraw (0.5,0) circle (0.06cm);
\draw [thick] (-0.5, 0) -- (0.5, 0);
\draw [thick] (0,-1) -- (0,0.5);
\draw [thick] (-0.5,0.5) -- (0.5, 0.5);
\draw [thick] (0,0.5) -- (0,1);
\draw [thick] (0.5, 0) -- (1,0);
\draw [thick] (0.5, 0) -- (0.5,-0.5);
\draw [thick] (-0.5, 0) -- (-1,0);
\draw [thick] (-0.5, 0) -- (-0.5,-0.5);
\end{tikzpicture}
\end{minipage}
\caption{Nested minimizers of the vertex-isoperimetry problem (left) and the universal comparison graph over the same vertex set (right).} 
\label{fig:ucg}
\end{figure}
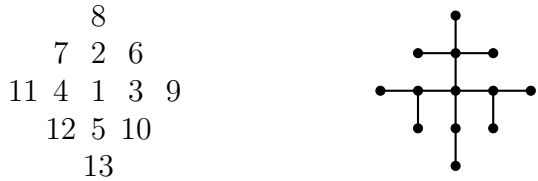

\begin{lemma} \label{lem:b} Suppose $(m,n)$, with $m < n$, is an edge in $(\mathbb{Z}^2, \ell^1)$ (where $m,n$ are integers and the corresponding vertices are with respect to the spiral labeling). Then
$$ n \leq m + 7\sqrt{m}.$$
 \end{lemma}
\begin{proof}
The vertex $m$ is on the boundary of an $\ell \times \ell$ square where $\ell = \left\lfloor \sqrt{m} \right\rfloor$. The smallest integer at the boundary of an $\ell \times \ell$ square is $(\ell - 1)^2 + 1$. Thus $m \geq (\ell - 1)^2 + 1$. Any neighbor has to be contained in an $(\ell + 2) \times (\ell + 2)$ square. The largest number arising in such a square is $(\ell + 2)^2$. Thus $n \leq (\ell + 2)^2$. We have
$$ n \leq (\ell + 2)^2 \leq (\ell - 1)^2 + 1 + 7 \sqrt{(\ell - 1)^2 + 1 }$$
for all $\ell \geq 10$. The cases $\ell \leq 10$ can be verified by hand.\end{proof}
We note that the argument is tight because $(1,8)$ is indeed an edge in $\mathbb{Z}^2$ with respect to the spiral labeling. It seems that one could asymptotically improve the constant for larger $m$ but it is not entirely clear how this could be leveraged into a better result.

\begin{lemma}\label{lem:c} $\Psi$ maps edges in $E_{(\mathbb{Z}^2, \ell^1)}$  to paths of length at most $4$. 
\end{lemma}
\begin{proof} Let us pick an edge $(m,n) \in (\mathbb{Z}^2, \ell^1)$ where $m < n$ and the numbering refers to the spiral labeling of $(\mathbb{Z}^2, \ell^1)$. We use
$r = d_{G_c}(1,m)$ to denote the distance between $m$ and $1$ in the universal comparison graph. Lemma \ref{lem:diam} implies that
$$ 1 + 2(r-1)(r+1) \leq m \leq 1 + 2r(r+1).$$
The upper bound implies  
$$r \geq \frac{\sqrt{2m-1}-1}{2} = X.$$
Appealing once more to Lemma \ref{lem:diam} shows that we expect $ \geq 4X+4, 4X+8, 4X+12$ vertices at distance $r+1$, $r+2$ and $r+3$, respectively. This means there are at least
$$12X + 24  = 6 \sqrt{2m-1} + 18 \geq  8\sqrt{m} + 15 >7\sqrt{m}$$
 vertices at distance $r+1 \leq r+3$. Lemma \ref{lem:b} implies that
picking a descendant $k$ of $m$ at distance 4 ensures $k \geq n$. Thus $\Psi$ maps edges to a path of length at most 4. 
\end{proof}

\begin{lemma}\label{lem3.3}
Let $e \in E_c$ be an edge in $G_c$. Then there are at most 16 edges $(i, j) \in E$ in the lattice graph such that $e$ lies in the path $\Psi(i, j)$.
\end{lemma}

\begin{proof} 
Let $e \in E_c$ be an edge in the universal comparison graph (which is a tree). We recall that $\Psi$ maps edges $(i,j) \in E_{(\mathbb{Z}^2, \ell^1)}$ to a path from $i$ to $k$ where $k$ is a descendant of $i$ at distance at most 4. This means that if we take the shortest path from the edge $e$ to the root, we are bound to find the vertex $i$ at distance at most 4 from the edge $e$. Thus there are at most 4 different vertices $i$ that could be the source of a path containing $e$. Since each vertex $i \in \mathbb{Z}^2$ has 4 neighbors, there are at most 16 edges $(i,j) \in E$ that could be mapped to a path containing $e$.
\end{proof}

\section{Proof of  Theorem \ref{thm1.5}}\label{sec:mainproof}
\begin{proof}[Proof of Theorem \ref{thm1.5}]
Let $f^*$ be the rearrangement of $f$ with respect to the spiral labelling on $\Z^2$ and let $f_c$ be the comparison function of $f$ on the universal comparison graph of the lattice graph. Consider an edge $(i, j) \in E$ and $\Psi(i, j)= \{x_0,x_1,..,x_n\}$ be a path in $G_c$ with $x_0<x_1<...<x_n$ and $n \leq 4$ (see Lemma \ref{lem:c}). Note that $x_0 = i$ and
$x_{n} \geq j$ from which we deduce $f^*(x_n) \leq f^*(j)$.
Jensen's inequality applied to $x \mapsto x^p$ for $p \geq 1$ implies that for $a_1, \dots, a_n > 0$
$$ \left( \frac{a_1 + \dots + a_n}{n} \right)^p \leq \frac{a_1^p}{n} + \dots +  \frac{a_n^p}{n}$$
and thus
$$ (a_1 + \dots + a_n)^p \leq n^{p-1} \left(a_1^p + \dots + a_n^p\right).$$
Therefore, using this together with $n\leq 4$ and the triangle inequality,
\begin{align*}
    |f^*(i)-f^*(j)|^p &\leq |f_c(x_0)-f_c(x_n)|^p\\
    &= \left| \sum_{k=0}^{n-1} f_c(x_k)-f_c(x_{k+1}) \right|^p \leq 4^{p-1} \sum_{k=0}^{n-1} |f_c(x_k)-f_c(x_{k+1})|^p.
\end{align*}
We now sum both sides of the inequality over all edges $(i,j) \in E_{(\mathbb{Z}^2, \ell^1)}$. Appealing to Lemma \ref{lem3.3}, we deduce that we end up summing over each edge in the universal comparison graph at most 16 times and thus, together with  Lemma \ref{lem2.5},
\begin{equation}\label{4.1}
    \left\|\nabla f^*\right\|_{L^p(\mathbb{Z}^2)}^p \leq 4^{p+1} \left\|\nabla f_c\right\|_{L^p(G_c)}^p \leq 4^{p+1}  \left\|\nabla f\right\|_{L^p(\mathbb{Z}^2)}^p.
\end{equation}
\end{proof}

\section{Proof of Theorem \ref{thm:wang}} \label{sec:wang}
\begin{proof} 
The proof is similar in style to the proof of Theorem \ref{thm1.5}, however, extremal properties of the Wang-Wang enumeration simplifies various steps.
 \begin{figure}[h!]
\begin{minipage}[l]{.3\textwidth}
\begin{tikzpicture}
\node at (-3,0) {};
\node at (0,0) {1};
\node at (0,0.5) {2};
\node at (0.5,0) {3};
\node at (-0.5,0) {4};
\node at (0,-0.5) {5};
\node at (0.5,0.5) {6};
\node at (-0.5,0.5) {7};
\node at (0,1) {8};
\node at (1,0) {9};
\node at (0.5,-0.5) {10};
\node at (-1,0) {11};
\node at (-0.5,-0.5) {12};
\node at (0,-1) {13};
\end{tikzpicture}
\end{minipage}
\begin{minipage}[l]{.53\textwidth}
\begin{tikzpicture}
\node at (-3,0) {};
\filldraw (0,0) circle (0.06cm);
\filldraw (0,0.5) circle (0.06cm);
\filldraw (-0.5,0) circle (0.06cm);
\filldraw (0,-0.5) circle (0.06cm);
\filldraw (0.5,0.5) circle (0.06cm);
\filldraw (-0.5,0.5) circle (0.06cm);
\filldraw (0,1) circle (0.06cm);
\filldraw (1,0) circle (0.06cm);
\filldraw (0.5,-0.5) circle (0.06cm);
\filldraw (-0.5,-0.5) circle (0.06cm);
\filldraw (-1,0) circle (0.06cm);
\filldraw (0,-1) circle (0.06cm);
\filldraw (0.5,0) circle (0.06cm);
\draw [thick] (-0.5, 0) -- (0.5, 0);
\draw [thick] (0,-1) -- (0,0.5);
\draw [thick] (-0.5,0.5) -- (0.5, 0.5);
\draw [thick] (0,0.5) -- (0,1);
\draw [thick] (0.5, 0) -- (1,0);
\draw [thick] (0.5, 0) -- (0.5,-0.5);
\draw [thick] (-0.5, 0) -- (-1,0);
\draw [thick] (-0.5, 0) -- (-0.5,-0.5);
\end{tikzpicture}
\end{minipage}
\caption{Nested minimizers of the vertex-isoperimetry problem (left) and the universal comparison graph over the same vertex set (right).} 
\label{fig:ucg2}
\end{figure}
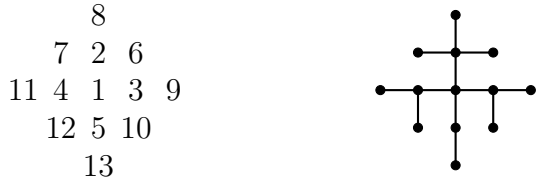

We map everything into the universal comparison tree and use Lemma \ref{lem2.5}
$$ \| \nabla f \|_{L^p(\mathbb{Z}^2)} \geq \| \nabla f_c\|_{L^p(G_c)}.$$
Note that the Wang-Wang construction uniformly minimizes the
vertex-boundary, the neighbors of $\left\{v_1, \dots, v_n\right\}$ in $\mathbb{Z}^2$
are exactly the same neighbors as the neighbors of $\left\{1, \dots, n\right\}$ in $G_c$:
they are given by 
\begin{equation}\label{6.1}
    \partial_V(\left\{v_1, \dots, v_n\right\}) = \left\{v_{n+1}, v_{n+2}, \dots, v_{n+\partial_V^n} \right\}
\end{equation}
in both cases. For $G_c$ this follows by construction (see Lemma \ref{lem2.4}), for $\mathbb{Z}^2$ with the
Wang-Wang enumeration this follows from the fact that the solutions are \textit{nested}: this means (see, for example, Bezrukov \& Serra \cite{bez}) that
\begin{enumerate}
\item there exists a nested sequence of sets of vertices
$$ A_1 \subset A_2 \subset A_3 \subset \dots$$
such that $\# A_i = i$ and $A_i$ is adjacent to as few vertices as possible for a set of vertices with $i$ elements and, moreover,
\item such that for all $i \in \mathbb{N}$ there exists $j \in \mathbb{N}$ such that $A_i \cup \partial_V(A_i) = A_j$.
\end{enumerate}

The main difference (see also Fig. \ref{fig:ucg2}) is that $\mathbb{Z}^2$ contains
some edges that are not contained in $G_c$. Using the Wang-Wang enumeration, it suffices to consider all edges $(i,j) \in \mathbb{Z}^2$ with $\|i\|_{\ell^1} + 1 = \|j\|_{\ell^1} =: r \geq 1$. There are two cases:
\begin{enumerate}
    \item $j$ is a corner point of $\ell^1$ ball of radius $r$. There is exactly one neighbour $i$ of $j$ in the $\ell^1$ ball of radius $r-1$ with $\|i\|_{\ell^1} = r-1$. Then using \eqref{6.1} we get,
    \begin{equation}\label{6.2}
        i+ \partial_V^{i-1} \leq j \leq i + \partial_V^i.
    \end{equation}
    Inequality \eqref{6.2} with the Definition \ref{def2.1} of comparison graph proves that $j$ is a neighbour of $i$ in $G_c$. 
    
    \item $j$ is not a corner point of $\ell^1$ ball of radius $\|j\|_{\ell^1}$. It is easy to see that $j$ will have exactly two neighbours in the ball of radius $r-1$ say $i_1, i_2$ with $\|i_1\|_{\ell^1} = \|i_2\|_{\ell^1} = r-1$. W.l.o.g. assume that $i_1 < i_2$. This shows
    \begin{equation}
        i_1 + \partial_V^{i_1-1} \leq j \leq i_1 + \partial_V^{i_1}
    \end{equation}
    Similarly, this proves that $j$ is connected to $i_1$ and not $i_2$ in $G_c$.   
\end{enumerate}
This proves that the comparison graph $G_c$ is obtained from $\Z^2$ by removing some of the edges between two consecutive $\ell^1$ spheres. 
The mapping 
$$ \Psi : E_{(\mathbb{Z}^2, \ell^1)} \rightarrow \{\text{paths in}\hspace{3pt} G_c\}$$
is now very simple: if $(i,j) \in E_{(\mathbb{Z}^2, \ell^1)}$, then we map $(i,j)$ to $(i,j)$ if that edge happens
to be in $G_c$. If not, then there exists exactly one $(k,j) \in G_c$ with $k < i$ and we map the edge to that.
In particular, comparing to the previous proof, $\Psi$ is mapping edges to edges and no application of Jensen's inequality is needed. 
Each edge in $G_c$ has a pre-image of cardinality at most 2 under $\Psi$ and thus, summing over all edges,
$$ \| \nabla f_c \|^p_{L^p(\mathbb{Z}^2, \ell^1)} \leq 2 \cdot \| \nabla f_c \|^p_{L^p(G_c)} \leq 2\cdot \| \nabla f \|^p_{L^p(\mathbb{Z}^2)}.$$

\end{proof}

\section{Proof of Theorem \ref{thm:abstract} and Corollary \ref{thm:cor}}
\label{sec:last}

\begin{proof}[Proof of Theorem \ref{thm:abstract}] The proof follows the same steps that we followed in the proof of the two-dimensional results. Much of the combinatorial complexity comes from explicitly bounding various constants which here are encapsulated in a more abstract condition. We start by returning to the notion of the universal comparison graph developed in \S 2.
We can apply Lemma 2.5 to deduce that
$$  \| \nabla f\|_{L^p(G)} \geq  \| \nabla f\|_{L^p(G_c)}$$
where $G_c$ is the universal comparison graph. Note that, as before, $G_c$ is an infinite tree with no leaves. It remains to compare edges in $G$ with short paths in $G_c$ just as we did before. As before, we now assume that we are given an enumeration of the vertices $v_1, v_2, \dots$ satisfying all assumptions of the Theorem.
We consider again the map
$$ \Psi : E_{G} \rightarrow \{\text{paths in}\hspace{3pt} G_c\},$$
defined in the same way as above: if $i<j$ and $(v_i,v_j) \in E(G)$ be an edge in the graph, then $\Psi(v_i,v_j)$ is defined as the shortest path in the tree $G_c$ between the vertex $v_i \in V_c$ and smallest vertex $v_k \in V_c$ such that
\begin{enumerate}
\item $k \geq j$ 
\item and $v_i$ lies in the shortest path between $v_1$ and $v_k$ in $G_c$. 
\end{enumerate}
Suppose now that $(v_i, v_j) \in E(G)$. The assumption in the Theorem says that the neighbors of $\left\{v_1, \dots, v_i\right\}$ are not too many, the set
is uniformly close to a vertex-isoperimetric set (up to a constant $c$). More precisely, since $v_j$ is adjacent to the set and $j > i$, we deduce that
$$ j \leq i + c \cdot \partial_V^i.$$
Mentally fixing $v_1$ as the root of the infinite tree, we can now pick the vertex $v_i$. The distance between these two vertices is
$r = d_{G_c}(v_1, v_i)$. By construction of the universal comparison graph, we have that 
$$ \# \left\{v \in G_c: d(v, v_1) = r+1\right\} \geq \partial_V^{i}.$$
Using the monotonicity of $\partial_V^n$, we deduce that the same inequality holds for larger `spherical shells' $ \# \left\{v \in G_c: d(v, v_1) = r+\ell\right\}$ and any
$\ell \geq 1$. By summation,
 $$\# \left\{v \in G_c:  r+1 \leq d(v, v_1) \leq r+\ell\right\} \geq \ell \cdot \partial_V^i.$$
 This shows that there is a descendant $v_k$ of $v_i$ satisfying $d_{G_c}(v_i, v_k) \leq c+1$ as well as $k \geq j$. Thus $\Psi$ is mapping edges to paths with length bounded by $c+1$. This means, appealing to both the tree structure and the construction of $\Psi$ that each edge $e \in G_c$ can only appear as the image of $\Psi$
 of an edge that is adjacent to at most $c+1$ different vertices in $G$. Since each vertex in $G$ has degree bounded from above by $D$, there are at most $(c+1) \Delta$ edges in question. Jensen's inequality implies, as above,
\begin{align*}
    |f^*(i)-f^*(j)|^p \leq (c+1)^{p-1} \cdot \sum_{k=0}^{n-1} |f_c(x_k)-f_c(x_{k+1})|^p,
\end{align*}
since $n \leq c+1$. Summing again both sides over all edges, we deduce
$$ \| \nabla f^*\|^p_{L^p} \leq (c+1)^{p} \Delta \cdot  \| \nabla f\|^p_{L^p}.$$
\end{proof}

\begin{proof}[Proof of Corollary \ref{thm:cor}]
The argument appeals to Theorem \ref{thm:abstract}. We first note that $\partial_V^{n+1} \geq \partial_V^n$ remains true and can be shown just as in $d=2$ (Lemma \ref{lem2.6}). We first recall a well-known isoperimetric inequality (see \cite{boll, boll2, wang}) telling us that for some universal constant $\alpha_d>0$ depending only on the dimension
$ \partial_{(\mathbb{Z}^d,\ell^1)}^n \geq \alpha_d \cdot n^{\frac{d-1}{d}}.$
The second observation is that if $A \subset \mathbb{Z}^d$ is a set such that 
$$  \exists k \in \mathbb{N} ~\forall~a \in A \qquad \|a\|_{\ell^1} \in \left\{k, k+1 \right\} $$
then $ \partial_V(A) \leq C_d (\#A)^{\frac{d-1}{d}}$. This follows from the fact that one can control the volume of a $1-$neighborhood of a convex domain in $\mathbb{R}^n$ from above by a multiple of its surface area (provided the convex domain is not too small in volume) and a simple comparison argument. Since $\alpha_d$ and $C_d$ are both universal constants, we can apply Theorem \ref{thm:abstract} to deduce the desired result.
\end{proof}

\textbf{Acknowledgment.}
The first author would like to thank Ari Laptev and Ashvni Narayanan for various useful discussions.

%%%%%%%%%%%%%%%%%%%%%%%%%%%%%%%%%%%%%%%%%%%
%%%%%%%%%%%%%%%%%%%%%%%%%%%%%%%%%%%%%%%%%%%


\begin{thebibliography}{67}

\bibitem{baernstein} A. Baernstein, Symmetrization in Analysis, Cambridge University Press, 2019. %\vspace{3pt}

\bibitem{bez} S.  Bezrukov and O. Serra, A local–global principle for vertex-isoperimetric problems,  Discrete Mathematics 257 (2002) 285--309


\bibitem{boll} B. Bollob\'as and I. Leader. Compressions and isoperimetric inequalities. {\em J. Comb. Theory, Ser. A},
\textbf{56} (1991), p. 47--62. %\vspace{3pt}

\bibitem{boll2} B. Bollob\'as and I. Leader. Edge-isoperimetric inequalities in the grid. {\em Combinatorica} \textbf{11} (1991), p. 299--314. %\vspace{3pt}

\bibitem{gar} J. Garcia-Domingo and J. Soria, A decreasing rearrangement for functions on homogeneous trees, {\em European Journal of Combinatorics} \textbf{26} (2005): 201--225. %\vspace{3pt}

\bibitem{gupta2022symmetrization} S. Gupta, Symmetrization inequalities on one-dimensional integer lattice. {\em arXiv:2204.11647}. %\vspace{3pt}

\bibitem{hajaiej2010rearrangement} H. Hajaiej, Rearrangement inequalities in the discrete setting and some applications. {\em Nonlinear Analysis: Theory, Methods \& Applications}. \textbf{72}, 1140-1148 (2010). %\vspace{3pt}

\bibitem{hajaiej2022discrete} H. Hajaiej, F. Han and B. Hua, Discrete Schwarz rearrangement in lattice graphs. {\em arXiv:2209.01003}.%\vspace{3pt}

\bibitem{hardy} G. H. Hardy, J. E. Littlewood and G. Polya, Inequalities, Cambridge Univ. Press, 1964. %\vspace{3pt}

\bibitem{harper} L. Harper, Optimal assignment of numbers to vertices. {\em J. SIAM} \textbf{12} (1964), p. 131--135. %\vspace{3pt} 

\bibitem{lieb1977existence} E. Lieb, Existence and uniqueness of the minimizing solution of Choquard's nonlinear equation. {\em Studies In Applied Mathematics}. \textbf{57}, 93-105 (1977). %\vspace{3pt}

\bibitem{lieb_hardy} E. Lieb, Sharp constants in the Hardy-Littlewood-Sobolev and related inequalities. {\em Ann. Of Math. (2)}. \textbf{118}, 349-374 (1983). %\vspace{3pt}

\bibitem{lieb2001analysis} E. Lieb and M. Loss, Analysis, American Mathematical Society, Providence, 2001. %\vspace{3pt}

\bibitem{lindsey} J. H. Lindsey. Assigment of numbers to vertices. {\em Amer. Math. Monthly} ,\textbf{71}:508--516,
1964. %\vspace{3pt}

\bibitem{Luttinger_schrodinger} J. Luttinger, Generalized isoperimetric inequalities. {\em J. Math. Phys.}. \textbf{14} pp. 586-593 (1973). %\vspace{3pt}

\bibitem{moser1971sharp} J. Moser, A sharp form of an inequality by N. Trudinger. {\em Indiana University Mathematics Journal}. \textbf{20}, 1077-1092 (1971). %\vspace{3pt}

\bibitem{faber_szego} G. Pólya and G. Szegö, Isoperimetric Inequalities in Mathematical Physics. Princeton University Press, Princeton, N. J.,1951. %\vspace{3pt}

\bibitem{pruss1998discrete} A. Pruss,  Discrete convolution-rearrangement inequalities and the Faber-Krahn inequality on regular trees. {\em Duke Mathematical Journal}. \textbf{91}, 463-514 (1998).% \vspace{3pt}

\bibitem{steinerberger2022discrete} S. Steinerberger, Discrete Rearrangements and the Polya-Szego Inequality on Graphs. {\em arXiv:2209.06765}. %\vspace{3pt}

\bibitem{talenti1976best}G. Talenti, Best constant in Sobolev inequality. {\em Ann. Mat. P. Appl}. \textbf{110}, 353-372 (1976). %\vspace{3pt}

\bibitem{wang} D. Wang and P. Wang, Discrete isoperimetric problems. {\em SIAM Journal On Applied Mathematics}. \textbf{32}, 860-870 (1977).% \vspace{3pt}

\end{thebibliography}
\end{document}